\documentclass[draft]{amsart}
\usepackage{amsfonts,amssymb,amsmath,amsthm}
\usepackage{enumerate}
\usepackage{soul} 

\newtheorem{theorem}{Theorem}[section]
\newtheorem{lemma}[theorem]{Lemma}
\newtheorem{proposition}[theorem]{Proposition}
\newtheorem{corollary}[theorem]{Corollary}
\theoremstyle{definition}
\newtheorem{definition}[theorem]{Definition}
\newtheorem{remark}[theorem]{Remark}
\numberwithin{equation}{section}

\author{P.~M.~Gadea}
\address{
Instituto de F\'\i sica Fundamental\\
CSIC\\
Serrano 113-bis, 28006-Ma\-drid, Spain}
\email{p.m.gadea@csic.es}

\author{J.~C.~Gonz{\'a}lez-D{\'a}vila}
\address{
Departamento de Matem\'aticas, Estad\'\i stica e Investigaci\'on Operativa\\
Universidad de La Laguna\\ 38200-La Laguna, Tenerife, Spain}
\email{jcgonza@ull.es}

\author{J.~A.~Oubi\~na}
\address{
Departamento de Xeometr\'{\i}a e Topolox\'\i a,
Facultade de Matem\'a\-ticas,
Universidade de Santiago de Compos\-tela,
15782-Santiago de Compostela, Spain}
\email{ja.oubina@usc.es}

\begin{document}

\title[Homogeneous spin Riemannian manifolds]{Homogeneous spin Riemannian manifolds with the simplest Dirac operator}

\begin{abstract}
We show the existence of nonsymmetric homogeneous spin Riemannian manifolds
whose Dirac operator is like that on
a Riemannian symmetric spin space. Such manifolds are exactly the homogeneous spin Riemannian manifolds $(M,g)$ 
which are traceless cyclic with respect to some quotient expression $M=G/K$ and reductive decomposition $\mathfrak{g} = \mathfrak{k} \oplus \mathfrak{m}$. 
Using transversally symmetric fibrations of noncompact type, we give a list of them.
\end{abstract}

\maketitle

\noindent {\scriptsize {\it Keywords:} Dirac operator, homogeneous spin Riemannian manifolds, traceless
cyclic homogeneous Riemannian manifolds, Riemannian symmetric spaces. \newline
{\it Acknowledgments:} The first and second authors have been supported by DGI (Spain) Project MTM2013-46961-P.}

\section{Introduction} 
\label{sect1}

The Dirac operator $D$ on a Riemannian 
symmetric spin space $(M=G/K,g)$, $(G,K)$ being a Riemannian symmetric pair,
may be written (cf.\ \cite[pp.\ 179--180]{Ike}, \cite[p.\ 230]{KobNom}) as
\begin{equation}
\label{d0ss1}
D\psi = \sum_{i=1}^n X_i \cdot X_i(\psi),
\end{equation}
where $(X_1,\dots,X_n)$ is a positively oriented orthonormal basis of the $-1$-eigenspace $\mathfrak{m}$ of the
involutive automorphism on the Lie algebra $\mathfrak{g}$ of $G$  determined by $(G,K)$. 

Actually, Ikeda \cite{Ike} gave a general formula for $M = G/K$ being a homogeneous Riemannian manifold, 
with $G$ unimodular and $K$ compact. In turn, B\"ar \cite[Theorem~1]{Bar}
extended formula \eqref{d0ss1} for $G$ being not necessarily unimodular.

The main purpose of this paper is to show the existence of nonsymmetric homogeneous spin Riemannian manifolds
$(M = G/K,g)$, with $G$ semisimple and $K$ connected, whose Dirac operator is expressed as in formula \eqref{d0ss1},
and to give a list of them. 

To this end, we recall that a homogeneous Riemannian manifold $(M,g)$ is said \cite{GadGonOub2} to be {\it cyclic\/} if there exists a quotient expression $M = G/K$ 
and a reductive decomposition $\mathfrak{g} = \mathfrak{k} \oplus  \mathfrak{m}$ satisfying
\begin{equation}
\label{onetwo}
\mathop{\text{\LARGE$\mathfrak S$}\vrule width 0pt depth 2pt}_{XYZ}\,\langle [X,Y]_\mathfrak{m} ,Z\rangle = 0,\qquad X,Y,Z \in \mathfrak{m},
\end{equation}
where $\langle \cdot,\cdot \rangle$ denotes the $\mathrm{Ad}(K)$-invariant inner product on $\mathfrak{m}$ induced by $g$. 
If moreover $G$ is unimodular, $(M,g)$ is said to be {\it traceless cyclic}.

The study of cyclic and traceless cyclic homogeneous Riemannian manifolds was started by Tricerri and
Vanhecke \cite{TriVan} and Kowalski and Tricerri \cite{KowTri}. In the latter paper, the
classification of simply connected traceless cyclic homogeneous Riemannian manifolds 
of dimension $\leqslant 4$ was given. The present authors extended it to the cyclic case in \cite{GadGonOub2}. 
The first examples which are not cyclic metric Lie groups \cite{GadGonOub1} do appear in dimension four.

Using the mentioned B\"ar expression of the Dirac operator on homogeneous spin
Riemannian manifolds, we prove the following result.

\begin{theorem}
\label{dirac}
 A homogeneous spin Riemannian manifold $(M = G/K,g)$ has
Dirac operator like that on a Riemannian symmetric spin space if and only if it is
traceless cyclic.
\end{theorem}

Theorem \ref{dirac} yields us to view, in spin geometry, traceless cyclic homogeneous spin Riemannian manifolds as the simplest manifolds after Riemannian symmetric spin spaces.

In the compact case, the expression \eqref{d0ss1} for the Dirac operator on a homogeneous spin Riemannian
manifold characterizes the class of Riemannian symmetric spaces. Specifically, we have  

\begin{theorem}
\label{theo11}
Let $M = G/K$ be a homogeneous spin Riemannian manifold with $G$ semisimple compact and $K$ connected.
Then it has Dirac operator like that on a Riemannian symmetric spin space if and only if $(G,K)$ is a Riemannian symmetric pair. 
\end{theorem}

In the noncompact case, we find a large class of nonsymmetric homogeneous spin Riemannian manifolds
whose Dirac operator admits an expression like \eqref{d0ss1} with respect to a suitable reductive decomposition.
Let $G$ be a connected Lie group and let $K$ and $L$ be closed subgroups of $G$ such that $K\subset L\subset G$.
Consider the natural projection $\pi\colon M = G/K\to N = G/L$, $gK\mapsto gL$. Then $\pi$ gives the homogeneous fibration
\[
F = L/K\to M = G/K \overset{\pi}{\to} N = G/L.
\]

We say that $\pi$ is a {\it transversally symmetric fibration\/} \cite{Gon} if $(G,L)$ is a Riemannian symmetric pair.
Moreover, $\pi$ is said to be of {\it compact type, noncompact type} or {\it Euclidean type\/} according to the type of $(G,L)$.
We will focus on the noncompact type. Let ${\mathfrak g} = {\mathfrak l} \oplus {\mathfrak p}$ be a Cartan decomposition,
where ${\mathfrak l}$ is the Lie algebra of $L$ and let $B$ be the Killing form of $G$. We assume that $K$ is compact and connected.
Then we have $B$-orthogonal reductive decompositions ${\mathfrak g} = {\mathfrak k} \oplus {\mathfrak m}$ and ${\mathfrak l}
= {\mathfrak k} \oplus {\mathfrak f}$, where ${\mathfrak m} = {\mathfrak f}\oplus {\mathfrak p}$.
A $G$-invariant metric on $M$ defined by an $\mathrm{Ad}(K)$-invariant inner product
on $\mathfrak{m}$ such that $\mathfrak{m} = \mathfrak{f} \oplus \mathfrak{p}$ is an orthogonal decomposition, 
is said to be {\it adapted\/} to the fibration. Due to Theorem \ref{dirac}, we seek for adapted (traceless) cyclic metrics. 

The triples $(G,L,K)$ of Lie groups where $G/K$ is a nonsymmetric cyclic homogeneous Riemannian manifold, with $K$ connected, fibering {over an irreducible Riemannian symmetric space $G/L$ (necessarily of noncompact type) and with isotropy-irreducible fibre type $L/K,$ are listed, for $G$ a classical simple Lie group, in Tables $1$ and $2$. For $G$ simple exceptional, the corresponding triples of Lie algebras $(\mathfrak{g},\mathfrak{l},\mathfrak{k})$ are given in Table $3$.

For transversally symmetric fibrations of noncompact type, we show that there exists a one-to-one correspondence between the sets of homogeneous spin Riemannian structures on the total space $M$ of the fibration $\pi$ and on the fibre type $F$. Concretely, we prove the following
\begin{theorem}
\label{conspi1}
Let $\pi\colon M = G/K \to N = G/L$ be a transversally symmetric fibration of noncompact type. Let $(\widetilde{G},\widetilde{\pi})$ be the universal covering of $G$ and $\widetilde{L} = \widetilde{\pi}^{-1}(L)$, $\widetilde{K} = \widetilde{\pi}^{-1}(K)$. The total space $M$, with a metric adapted to the fibration
$M = \widetilde{G}/\widetilde{K} \to \widetilde{G}/\widetilde{L}$, is a homogeneous spin Riemannian manifold 
with adapted  (see {\em Definition \ref{defhsrm})}  quotient $\widetilde{G}/\widetilde{K}$ if and only if the fibre type $F$, with
adapted quotient $\widetilde{L}/\widetilde{K}$,
is also a homogeneous spin Riemannian manifold. Moreover, the Dirac operator on $M$ is like that on a Riemannian symmetric spin space if and only if the pair $(L,K)$ is associated with an orthogonal symmetric Lie algebra.
\end{theorem}

As a direct consequence, if $F$ is simply connected in Theorem \ref{conspi1}, the (simply connected) homogeneous spin Riemannian manifold $(M = G/K,g)$ has Dirac operator like that on a Riemannian symmetric spin space if and only if $F = L/K$ is a (compact) Riemannian symmetric spin space.

We recall in Table $4$ the Cahen and Gutt \cite[Theorems 8 and 12]{CahGut}
classification of compact simply connected Riemannian symmetric spin spaces $L/K$
with $L$ simple, with a few minor changes. Then, using Theorem \ref{conspi1} and Proposition \ref{conspi}, we give many examples of nonsymmetric homogeneous spin Riemannian manifolds $G/K$ having Dirac operator like that on a Riemannian symmetric spin space, in Tables $1$, $2$ and (through the corresponding pairs $(\mathfrak{g},\mathfrak{k}))$ in Table $3$. 

The paper is organized as follows. In Section~\ref{secthr} the notion of homogeneous spin Riemannian manifold is recalled and, 
after giving B\"ar's expression of the Dirac operator on homogeneous spin Riemannian
manifolds, we prove Theorem \ref{dirac}.
In Section~\ref{secfou}, we study nonsymmetric cyclic homogeneous Riemannian manifolds $M=G/K$ with $G$
semisimple. 
If $G$ is compact and $K$ connected, from Theorem \ref{dirac} above and Proposition \ref{psemisimple}, Theorem \ref{theo11} follows. 
Then we suppose that $G$ is not compact and $K$ is a connected compact subgroup of $G$ (see Proposition \ref{noncomsemsim}) and we obtain Tables $1$, $2$ and $3$ (up to the fifth column in Tables $1$ and $3$,
and up to the fourth column in Table $2$). 
In Section~\ref{secfiv}, Theorem \ref{conspi1} is proven, so having 
examples of nonsymmetric homogeneous spin Riemannian manifolds $G/K$ having Dirac operator like that on a Riemannian symmetric spin space. In 
Section~\ref{secsix}, we give a detailed example, checking that the nonsymmetric manifold
$\widetilde{SL(2,\mathbb{R})}$, with each metric of a biparametric family of traceless cyclic
metrics, has Dirac operator like that on a Riemannian symmetric spin space.

\section{The Dirac operator on traceless cyclic homogeneous spin Riemannian manifolds}
\label{secthr}

Let $(M,g)$ be an $n$-dimensional spin Riemannian manifold and let $\mathrm{Spin}(n) \hookrightarrow P \to M$ be a fixed spin structure on $M$. 
Denote by $\rho$ the basic spin representation of $\mathrm{Spin}(n)$.
Let $\Sigma(M)$ be the corresponding spinor bundle over $M$ associated with $P$, that is,
\[
\Sigma(M) = P \times_{\rho} \Delta,
\]
$\Delta = \mathbb{C}^{2^{[n/2]}}$ being the representation space of $\rho$. 
Spinor fields are defined as sections $\psi\colon M \to \Sigma(M)$ or, equivalently, as maps $P \to \Delta$ which are
equivariant with respect to the action of $\mathrm{Spin}(n)$, that is,
$\psi(za) = \rho(a^{-1})\psi(z)$, $z\in P$, $a\in \mathrm{Spin}(n)$.

The Levi-Civita connection $\nabla$ on $(M,g)$ naturally induces \cite[Chapter II, Theorem 4.14]{LawMic} a 
connection $\nabla^\Sigma$ on $\Sigma(M)$, which may be described
\cite[p.\ 224]{AmmBar} as follows. Let $(e_1,\dots,e_n)$
be a positively oriented local orthonormal frame defined on a connected open subset $U\subset M$.
Then $(e_1,\dots,e_n)$ is a local section of the bundle of
positively oriented orthonormal frames $\mathcal{SO}(M)$.
Denote by $\hat{\lambda}$ the projection map $\hat{\lambda}\colon P \to \mathcal{SO}(M)$, and let $\ell$ be a lift to $P$ such that $\hat{\lambda} {\makebox[7pt]{\raisebox{1.5pt}{\scriptsize $\circ$}}} \ell = (e_1,\dots,e_n)$. Then $\ell$ defines a trivialization
$U \times \Delta \cong \Sigma(M)|_U$ of $\Sigma(M)$ over $U$, with respect to which one has the following formula for $\nabla^\Sigma$,
\begin{equation}
\label{nablasig}
\nabla^{\Sigma}_{e_i}\psi = e_i(\psi) + \frac14 \sum_{j,k=1}^n \Gamma^k_{ij} e_j \cdot e_k \cdot \psi,
\end{equation}
where $\psi \in \Gamma(\Sigma(M))$, $e_i(\psi)$ denotes differentiation of $\psi$ by $e_i$,
$\Gamma^k_{ij}$ are the corresponding Christoffel symbols with respect to $(e_1,\dots,e_n)$ and
$e_j \cdot e_k \cdot \psi := \rho(e_j)$ $\rho(e_k)\psi$.

The Dirac operator $D \colon \Gamma(\Sigma(M)) \to \Gamma(\Sigma(M))$ is then defined by
\begin{equation}
\label{dir}
D\psi = \sum_{i=1}^n e_i \cdot \nabla^{\Sigma}_{e_i}\psi.
\end{equation}

A connected homoge\-ne\-ous Riemannian manifold
$(M,g)$ can be described as a coset manifold $G/K$, where $G$ is a
Lie group, which we suppose to be connected, acting transitively and effectively by isometries
on $M$, $K$ is the isotropy subgroup of $G$ at
some point $o\in M$, the origin of $G/K$, and $g$ is a
$G$-invariant Riemannian metric. Moreover, we can assume  that $G$ is a closed subgroup of the full isometry group $I(M,g)$ of $(M,g)$. Then, $K$ is a compact subgroup and $G/K$ admits a reductive decomposition, that is,  there is an
$\mathrm{Ad}(K)$-invariant subspace $\mathfrak{m} $ of the Lie algebra $\mathfrak{g}$ of $G$
such that $\mathfrak{g}$ splits as the vector space direct sum $\mathfrak{g}  = \mathfrak{k} \oplus  \mathfrak{m} $,
$\mathfrak{k}$ being the Lie algebra of $K$.

Next, suppose that $(M = G/K,g)$ is an oriented homogeneous Riemannian manifold with a fixed reductive decomposition $\mathfrak{g}  = \mathfrak{k} \oplus  \mathfrak{m}$. By using the identification $T_{o}(M) \cong \mathfrak{m}$ and taking into account that $G$ is connected, 
the isotropy representation 
$\chi$ may be expressed as the map~$\chi\colon K \to SO(\mathfrak{m})$, $\chi(k)(X) = \mathrm{Ad}_{k}(X)$, for all $k\in K$ and $X\in \mathfrak{m}$. Then the
choice of a positively oriented orthonormal basis $(X_1,\dots,X_n)$ of $\mathfrak{m}$ gives us the identification 
\[
G\times_{\chi}SO(\mathfrak{m}) \equiv \mathcal{SO}(M), 
\]
via the map $[(a,A)] \mapsto (\tau_{a}(o);(\tau_{a})_{*o}AX_{1},\dotsc ,(\tau_{a})_{*o}AX_{n})$.  
Let $\lambda\colon \mathrm{Spin}(\mathfrak{m}) \to SO(\mathfrak{m})$
be the usual two-sheet covering map.
 If there exists a Lie group homomorphism $\widetilde\chi \colon K \to \mathrm{Spin}(\mathfrak{m})$ lifting $\chi$, that is, $\lambda \circ \widetilde\chi = \chi$, we can define a spin structure $P_{\widetilde\chi}$ on $M$ by 
\[
P_{\widetilde\chi} = G\times_{\widetilde\chi} \mathrm{Spin}(\mathfrak{m}), 
\] 
 with covering map $\hat{\lambda}\colon\! P_{\widetilde\chi}\to \mathcal{SO}(M)= G\times_{\chi} SO(\mathfrak{m})$, given by $\hat{\lambda}([(a,\!A)] = [(a,\!\lambda(A))]$.

Throughout this paper, we consider homogeneous spin Riemannian manifolds as in the next (well-known) definition.  
\begin{definition}
\label{defhsrm}
An oriented connected homogeneous Riemannian manifold $(M,g)$ of dimension $n$ is said to be a {\it homogeneous spin Riemannian manifold\/} if there exists a quotient expression $M = G/K$ and a reductive decomposition $\mathfrak{g} = \mathfrak{k} \oplus \mathfrak{m}$ such that the isotropy representation $\chi$ can be lifted to a Lie group homomorphism $\widetilde{\chi}\colon K \to \mathrm{Spin}( \mathfrak{m})$. The spin structure 
$\mathrm{Spin}(n) \to P_{\widetilde\chi} = G \times_{\widetilde\chi} \mathrm{Spin}(\mathfrak{m}) \to M$ on $M$ is called a {\it homogeneous spin structure}. Then $G/K$ (resp., $\mathfrak{g} = \mathfrak{k} \oplus \mathfrak{m})$ is said to be  a quotient expression (resp., reductive decomposition) {\it adapted\/} to the homogeneous spin structure.
\end{definition}

A homogeneous Riemannian manifold $(M = G/K,g)$ with $G$ simply connected and admitting a spin structure is a homogeneous spin Riemannian manifold. Concretely, we have the following
\begin{lemma}
\label{lBar} 
{\rm \cite[Lemma 3]{Bar}} If $G$ is simply connected, there exists a one-to-one correspondence between the set of spin structures of the oriented manifold $(M = G/K,g)$ and the set of lifts $\widetilde\chi$ of $\chi$.
\end{lemma}

Consider the universal covering $(\widetilde{G},\widetilde{\pi})$ of a connected Lie group $G$. Then $\widetilde{\pi}\colon \widetilde{G}\to G$ is a Lie group homomorphism. Put $\widetilde{K}=\widetilde{\pi}^{-1}(K)$ and denote by $\mathrm{Ad}^{\widetilde{G}}$ the adjoint representation of $\widetilde{G}$.

\begin{lemma}
\label{luniversal} 
Let $(G/K,g)$ be a homogeneous Riemannian manifold with reductive decomposition $\mathfrak{g} = \mathfrak{k}\oplus \mathfrak{m}$ and $G$-invariant metric $g$ determined by an $\mathrm{Ad}(K)$-invariant inner product $\langle\cdot,\cdot\rangle$ on $\mathfrak{m}$. If $\mathfrak{m}$ and $\langle\cdot,\cdot\rangle$ are $\mathrm{Ad}^{\widetilde{G}}(\widetilde{K})$-invariant, then $(G/K,g)$ is naturally isometric to $(\widetilde{G}/\widetilde{K},\tilde{g})$, where $\tilde{g}$ is the $\widetilde{G}$-invariant metric determined by $\langle\cdot,\cdot\rangle$.
\end{lemma}
\begin{proof} 
Because $\widetilde{G}$ acts transitively on the left on $G/K$ via the mapping $(\tilde{a},\widetilde{\pi}(\tilde{b})K)\mapsto \widetilde{\pi}(\tilde{a}\tilde{b})K$, for all $\tilde{a},\tilde{b}\in \widetilde{G}$, we have that the map $j\colon \widetilde{G}/\widetilde{K}\to G/K$, defined by $j(\tilde{a}\widetilde{K}) = \widetilde{\pi}(\tilde{a})K$, is a diffeomorphism. Denote by $\tau_{\widetilde{\pi}(\tilde{a})}$ and $\tilde{\tau}_{\tilde{a}}$ the translations on $G/K$ and $\widetilde{G}/\widetilde{K}$, respectively. Then,
\[
\tau_{\widetilde{\pi}(\tilde{a})}{\makebox[7pt]{\raisebox{1.1pt}{\scriptsize $\circ$}}} j = j {\makebox[7pt]{\raisebox{1.1pt}{\scriptsize $\circ$}}}\tilde{\tau}_{\tilde{a}},\;\;\;\mbox{\rm for all}\;\tilde{a}\in \widetilde{G}.
\]
Denoting by $o$ and $\tilde{o}$ the corresponding origins of $G/K$ and $\widetilde{G}/\widetilde{K}$, the tangent spaces $T_{o}(G/K)$ and $T_{\tilde{o}}(\widetilde{G}/\widetilde{K})$ are identified with $\mathfrak{m}$ by the isomorphism $j_{*o}$. Hence, for all $u,v\in T_{\tilde{\tau}_{\tilde{a}}(\tilde{o})}(\widetilde{G}/\widetilde{K})$, we get 
\begin{align*}
\tilde{g}_{\tilde{\tau}_{\tilde{a}}(\tilde{o})}(u,v) & = \big\langle (\tilde{\tau}_{\tilde{a}^{-1}})_{*\tilde{\tau}_{\tilde{a}}(\tilde{o})}u,(\tilde{\tau}_{\tilde{a}^{-1}})_{*\tilde{\tau}_{\tilde{a}}(\tilde{o})}v\big\rangle \\[0.4pc]
& = \big\langle j_{*\tilde{o}}(\tilde{\tau}_{\tilde{a}^{-1}})_{*\tilde{\tau}_{\tilde{a}}(\tilde{o})}u,j_{*\tilde{o}}(\tilde{\tau}_{\tilde{a}^{-1}})_{*\tilde{\tau}_{\tilde{a}}(\tilde{o})}v\big\rangle\\[0.4pc]
 & = \big\langle (\tau_{\pi(\tilde{a}^{-1})})_{*\tau_{\pi(\tilde{a})}(\tilde{o})}j_{*\tilde{\tau}_{\tilde{a}}(\tilde{o})}u, (\tau_{\pi(\tilde{a}^{-1})})_{*\tau_{\pi(\tilde{a})}(\tilde{o})}j_{*\tilde{\tau}_{\tilde{a}}(\tilde{o})}v\big\rangle \\[0.4pc]
 & = g_{\tau_{\pi(\tilde{a})}(o)}(j_{*\tilde{\tau}_{\tilde{a}}(\tilde{o})}u, j_{*\tilde{\tau}_{\tilde{a}}(\tilde{o})}v). \qedhere
 \end{align*}  
\end{proof}

\begin{proposition}
\label{psimple}
A simply connected homogeneous Riemannian manifold $(M = G/K,$ $g)$ equipped with a spin structure is a homogeneous spin Riemannian manifold with adapted quotient expression $\widetilde{G}/\widetilde{K}$.
\end{proposition}
\begin{proof} 
It follows from Lemmas \ref{lBar} and \ref{luniversal}, using that $\widetilde{K}$ is connected.
\end{proof}

\begin{remark} 
\label{pnon-comp}
A (Riemannian) manifold $M$ admits a spin structure if and only if the Stiefel-Whitney classes $w_1(M)$ and $w_2(M)$ vanish. 
A Riemannian symmetric space $M$ of noncompact type is contractible \cite[Chapter XI, Theorem 8.6]{KobNom},
so its Stiefel-Whitney classes vanish and then it admits a spin structure. From Proposition \ref{psimple}, $M$ is a homogeneous spin Riemannian manifold.
\end{remark}

The spinor bundle $\Sigma_{\bar\chi}(M)$ of the homogeneous spin structure  $P_{\widetilde\chi}$ is given \cite[Lemma 4]{Bar} by 
\[
\Sigma_{\tilde\chi}(M) =  G\times_{\rho{\makebox[7pt]{\raisebox{0.1pt}{\scriptsize $\circ$}}} \widetilde\chi}\Delta. 
\]
Hence, spinor fields can be viewed as $\rho{{\makebox[7pt]{\raisebox{0.6pt}{\scriptsize $\circ$}}}} 
\widetilde\chi$-equivariant maps $\psi\colon G\to \Delta$.

The spinor connection $\nabla^{\Sigma}$ induced by the Levi-Civita connection of $(M,g)$ is given \cite{Bar} by 
\[
\nabla^{\Sigma}_{(\tau_{g})_{*o}X}\psi = \Big(X(\psi) + \frac{1}{4}\sum_{i,j} c_{ij}(X)X_{i}\cdot X_{j}\cdot \psi\Big) (g),
\]
for all $X\in \mathfrak{m}$, $g\in G$, where $X(\psi)(g) = (\mathrm{d}/\mathrm{d}t)_{\mid t = 0} (\psi(g\exp tX))$  and 
\[
c_{ij}(X) = \langle \nabla_{X_{i}}X^{*},X_{j}\rangle = \frac12 \{ -\langle [X_i,X]_\mathfrak{m} , X_j \rangle +
\langle [X_j,X_i]_\mathfrak{m}, X \rangle + \langle [X_j,X]_\mathfrak{m}, X_i \rangle \},
\]
where $X^{*}$ stands for the fundamental vector field associated to $X$, that is, $X^{*}_{p} = (\mathrm{d}/\mathrm{d}t)_{\mid t = 0}((\exp tX)p)$, for all $p\in M.$ In particular, if $(G,K)$ is a Riemannian symmetric pair then $\nabla^{\Sigma}_{(\tau_{g})_{*o}X}\psi = X(\psi)$ and the Dirac operator
may be written, in terms of a positively oriented orthonormal basis $(X_1,\dotsc,X_n)$ of $\mathfrak{m} $ (cf.\ \cite{Ike}), as in \eqref{d0ss1}. For homogeneous spin Riemannian manifolds $(M = G/K,g)$, with adapted reductive decomposition $\mathfrak{g} = \mathfrak{k}\oplus \mathfrak{m}$, the Dirac operator is then given by  
$$
\begin{array}{lcl}
(D\psi)(g) &  = & \sum_{k}(\tau_{g})_{*o}X_{k}\cdot (\nabla^{\Sigma}_{(\tau_{g})_{*o}X_{k}}\psi)\\[0.4pc]
 &= & \sum_{k}(X_{k}\cdot X_{k}(\psi) + \frac{1}{4}\sum_{i,j} c_{ijk} X_k \cdot X_{i}\cdot X_{j}(\psi))(g),
\end{array}
$$
where $c_{ijk} = c_{ij}(X_{k}).$ Hence, the determination of the  coefficients $c_{ijk}$ gives the B\"ar formula \cite[\S2]{Bar}, which, using that $[\mathfrak{k},\mathfrak{m}]\subset \mathfrak{m}$ and $K$ is compact, can be expressed as 
\begin{equation}
\label{barfor1}
\begin{split}
D\psi  & = \sum_{i=1}^n X_i\cdot X_i(\psi)
               + \frac14 \Bigl(\,\sum_{i<j<k} \Big(\mathop{\text{\LARGE$\mathfrak S$}
\vrule width 0pt depth 2pt}_{X_{i}X_{j}X_{k}}\, \langle[X_i,X_j]_{\mathfrak m},X_k\rangle \Big)
\,X_i\cdot X_j\cdot X_k  \\
         & \hspace{32mm}  - 2\sum_{i=1}^n (\mathrm{tr}\, \mathrm{ad}_{X_{i}})X_{i}\Big)
\cdot \psi.
\end{split}
\end{equation}
\begin{remark} 
Let $\mathrm{Cl}(\mathfrak{m})$ be the Clifford algebra of the negative form on $\mathfrak{m}$ and
let $\mathrm{Cl}_{\mathbb{C}}(\mathfrak{m})$ be its complexified Clifford algebra. On account of the maps
$\mathfrak{m} \hookrightarrow \mathrm{Cl}(\mathfrak{m}) \hookrightarrow \mathrm{Cl}_{\mathbb{C}}
(\mathfrak{m}) \overset{\rho}{\to} \mathrm{End}(\Delta)$,
each $X_i$ in \eqref{barfor1} acts on spinors
as an element of $\mathrm{End}(\Delta)$ $ = \mathrm{End}(\mathbb{C}^{2^{[n/2]}})$. 
\end{remark}

\noindent {\sf Proof of Theorem \ref{dirac}.} Suppose that $(M=G/K,g)$ is a traceless cyclic homogeneous spin Riemannian manifold with adapted reductive decomposition $\mathfrak{g} = \mathfrak{k} \oplus \mathfrak{m}$. Then the last two summands in \eqref{barfor1} vanish and we are left with the expression \eqref{d0ss1}, which is the one for a Riemannian symmetric spin space.

Conversely, let $(M=G/K,g)$ be a homogeneous spin Riemannian manifold and let $\mathfrak{g} = \mathfrak{k} \oplus \mathfrak{m}$ be an adapted reductive decomposition. Suppose that it has Dirac operator like that on
a Riemannian symmetric spin space. Then we have, on account of \eqref{d0ss1} and \eqref{barfor1}, that 
\[
\sum_{i<j<k} \Big(\mathop{\text{\LARGE$\mathfrak S$}\vrule width 0pt depth 2pt}_{X_{i}X_{j}X_{k}}
\,\langle[X_i,X_j]_{\mathfrak m},X_k\rangle \Big) \,X_i\cdot X_j\cdot X_k - 2\sum_{i=1}^n (\mathrm{tr}\, \mathrm{ad}_{X_{i}})X_{i} = 0. 
\] 
Now, if $n\geqslant 3$, then $X_i$, $X_j\cdot X_k\cdot X_l$, $i=1,\dotsc,n$, $1\leqslant j<k<l\leqslant n$,
are linearly independent elements of $\mathrm{Cl}_{\mathbb{C}}(\mathfrak{m})$, hence each
cyclic sum $\mathfrak{S}_{X_iX_jX_k}\langle[X_iX_j]_{\mathfrak m},X_k\rangle$ and each coefficient $\mathrm{tr}\, \mathrm{ad}_{X_{i}}$ is null.
It is immediate that $(M = G/K,g)$ is  traceless cyclic and $\mathfrak{g} = \mathfrak{k} \oplus \mathfrak{m}$ is an adapted reductive decomposition.

If $n=2$, it is immediate from \cite[Theorem 3.1]{TriVan} and \cite[Definition 4.1]{GadGonOub2} that $G/K$ is cyclic. Now, if $M$ has Dirac operator like that
on a Riemannian symmetric spin space, one has $\mathrm{tr}\, \mathrm{ad}_{X}=0$, for all $X\in {\mathfrak m}$, and hence $G$ is unimodular, then $M=G/K$ is
also traceless cyclic. \hfill $\square$

\section{Cyclic homogeneous Riemannian manifolds $G/K$ with $G$ semisimple}
\label{secfou}

Let $(M=G/K,g)$ be a homogeneous Riemannian manifold with $G$ semisimple and let $B$ be the Killing form
of $G$. Since the isotropy subgroup $K$ is compact, the restriction of $B$ to $\mathfrak{k}$ is negative definite and one gets the direct sum
$\mathfrak{g}=\mathfrak{k}\oplus\mathfrak{m} $, where $\mathfrak{m} $ is the $B$-orthogonal complement of
$\mathfrak{k}$ in $\mathfrak{g}$. Such decomposition is clearly reductive. Then each $G$-invariant metric
on $M$ is uniquely determined by an $\mathrm{Ad}(K)$-invariant inner product on $\mathfrak{m} $.

\begin{lemma}
\label{ldes}
Let $\mathfrak{m} = \mathfrak{f} \oplus \mathfrak{p}$ be a $B$-orthogonal decomposition of $\mathfrak{m} $
into two $\mathrm{Ad}(K)$-invariant subspaces $\mathfrak{f} $ and $\mathfrak{p}$ $(\mathfrak{f} $
or $\mathfrak{p}$ may be zero$)$. Suppose that the restrictions   $B_{|\mathfrak{f}\times \mathfrak{f}}$ and
$B_{|\mathfrak{p}\times \mathfrak{p}}$ are negative definite and positive definite, respectively.
Then, for any $\mathrm{Ad}(K)$-invariant inner product $\langle\cdot,\cdot\rangle$ on $\mathfrak{m} $,
making orthogonal the subspaces $\mathfrak{f}$
 and $\mathfrak{p}$, there is an orthogonal splitting
\begin{equation}
\label{split31}
\mathfrak{m}  = \mathfrak{f} _{1}\oplus \dots \oplus\mathfrak{f} _{r} \oplus \mathfrak{p}_{1}\oplus \dots \oplus \mathfrak{p}_{s},
\end{equation}
where $K$ acts irreducibly on $\mathfrak{f} _{i}$ and $\mathfrak{p}_{j}$, for $i = 1,\dotsc ,r$ and $j = 1,\dotsc ,s$,
by the adjoint map and $\langle\cdot,\cdot\rangle$ is a \emph{diagonal} direct sum with respect to $B$ of the form
  \begin{equation}
\label{inner}
  \langle\cdot,\cdot\rangle = \sum_{a=1}^{r}\lambda_{a}B_{|\mathfrak{f} _{a}\times\mathfrak{f} _{a}}
+ \sum_{b=1}^{s}\lambda_{r+b}B_{|\mathfrak{p}_{b}\times\mathfrak{p}_{b}},
\end{equation}
where $\lambda_{a}<0$  for $a = 1,\dotsc ,r$, and $\lambda_{r+b}>0$ for
$b = 1,\dotsc ,s$.
\end{lemma}
\begin{proof} 
Let $Q$ be the unique symmetric bilinear map on $\mathfrak{m} $ such that
\[
\langle x,y\rangle = B(Qx,y),\qquad x,y\in \mathfrak{m} .
\]
Using that $B$ is nondegenerate and $\langle\cdot,\cdot\rangle$ and $B$ are both $\mathrm{Ad}(K)$-invariant, one gets that $Q
{\makebox[7pt]{\raisebox{1.5pt}{\scriptsize $\circ$}}} \mathrm{Ad}_{k} = \mathrm{Ad}_{k}
{\makebox[7pt]{\raisebox{1.5pt}{\scriptsize $\circ$}}} Q$, for all $k\in K$. Hence, the eigenspaces
of $Q$ are $\mathrm{Ad}(K)$-invariant subspaces. They are mutually orthogonal for both
$\langle\cdot,\cdot\rangle$ and $B$. Moreover, on each irreducible $\mathrm{Ad}(K)$-invariant subspace, both 
$\langle\cdot,\cdot\rangle$ and $B$ must be proportional, with proportionality factor the
corresponding eigenvalue of $Q$. Since $\langle\cdot,\cdot\rangle$ is strictly definite on
$\mathfrak{f} $ and on $\mathfrak{p}$, such factors must be different from zero.
\end{proof}

\smallskip

Choose a $B$-orthogonal basis $\{e_{i}^{a} : 1\leqslant  i\leqslant n_{a}, 1\leqslant  a\leqslant r+s\}$
adapted to the splitting \eqref{split31},  where $n_{a} = \dim \mathfrak{f}_{a}$, if $a=1,\dotsc,r$,
and  $n_{r+b} = \dim \mathfrak{p}_{b}$, if $b=1,\dotsc ,s$.  This basis is $\langle\cdot,\cdot\rangle$-orthogonal. Put
 \[
[e^{a}_{i},e_{j}^{b}]_\mathfrak{m} = \sum_{1\leqslant c\leqslant r+s\atop 1\leqslant k\leqslant n_{c}}c_{i_{a}j_{b}}^{k_{c}}e_{k}^{c}.
\]
Then, arguing as in \cite[Proposition 8.2]{GadGonOub2}, the next result follows.

\begin{proposition}
\label{psemisimple} 
The $G$-invariant metric on $M$ determined by the inner product
{\rm(\ref{inner})} is {\rm (}traceless{\rm )} cyclic if and only if
\begin{equation}
\label{ccc}
c_{i_{a}j_{b}}^{k_{c}}(\lambda_{a} + \lambda_{b} + \lambda_{c}) =  0,
\end{equation}
where $a,b,c = 1,\dots ,r+s$, $1\leqslant i\leqslant n_{a}$,
$1\leqslant j\leqslant n_{b}$, $1\leqslant k\leqslant n_{c}$.
\end{proposition}

\noindent {\sf Proof of Theorem \ref{theo11}.} 
From Proposition \ref{psemisimple} we have that the pair $(G,K)$ is a Riemannian symmetric pair, since it is associated with an orthogonal
symmetric Lie algebra of compact type and $K$ is connected \cite[Chapter VII, Exercise 10]{Hel}. By using now Theorem \ref{dirac}, 
we obtain Theorem \ref{theo11}. 
  
\medskip

We now  consider cyclic homogeneous Riemannian manifolds $G/K$ with $G$ semi\-sim\-ple noncompact. More precisely, let $\pi\colon M = G/K\to N = G/L$ be a transversally symmetric fibration of noncompact type, with $K$ compact and connected. Consider the maximal compact subgroup
$L$ of $G$ such that $K\subset L$ and let $\mathfrak{g}=\mathfrak{l}  \oplus \mathfrak{p}$ be the Cartan decomposition,
where $\mathfrak{l} $ is the Lie algebra of $L$. Let $\mathfrak{m} $ be the direct sum $\mathfrak{m}  = \mathfrak{f}
\oplus \mathfrak{p}$, $\mathfrak{f} $ being the $B$-orthogonal complement to the Lie algebra $\mathfrak{k}$ of $K$
in $\mathfrak{l}$. Then, we have 
\begin{proposition}
\label{noncomsemsim} 
The homogeneous manifold $M = G/K$ admits a cyclic metric adap\-ted to the transversally symmetric fibration $\pi\colon M = G/K\to N = G/L$ if and only if the pair $(L,K)$ is associated to an orthogonal symmetric Lie algebra $(\mathfrak{l},s)$.
\end{proposition}

\begin{proof} 
Because $B$ is strictly definite on $\mathfrak{l} $ $(B<0)$ and on $\mathfrak{p}$ $(B>0)$, there are
$B$-orthogonal decompositions $\mathfrak{g}=\mathfrak{l}  \oplus \mathfrak{p}$, $\mathfrak{l} =\mathfrak{k}
\oplus \mathfrak{f} $,
    $\mathfrak{m} =\mathfrak{f}  \oplus \mathfrak{p}$ and $\mathfrak{g}=\mathfrak{k} \oplus \mathfrak{m} $, where
\[
      [\mathfrak{l} ,\mathfrak{p}]\subset \mathfrak{p}, \quad [\mathfrak{k},\mathfrak{f} ]\subset \mathfrak{f} ,
\quad [\mathfrak{k},\mathfrak{m} ]\subset\mathfrak{m} .
    \]
We also have $[\mathfrak{f} ,\mathfrak{p}]\subset\mathfrak{p}$ and $[\mathfrak{f} ,\mathfrak{f} ]_{\mathfrak{m} }\subset\mathfrak{f} $.
Since $(G,L)$ is a symmetric pair, one has $[\mathfrak{p},\mathfrak{p}]\subset\mathfrak{l} $, and hence
\[
       [\mathfrak{p},\mathfrak{p}]_{\mathfrak{m} } \subset (\mathfrak{k} \oplus \mathfrak{f} )_{\mathfrak{m} } =\mathfrak{f} .
   \]

Suppose first that there exists a cyclic metric on $G/K$ with respect to
$\mathfrak{g}=\mathfrak{k} \oplus \mathfrak{m} $ which
is defined by an $\mathrm{Ad}(K)$-invariant  inner product $\langle\cdot,\cdot\rangle$  on $\mathfrak{m} $ making
orthogonal $\mathfrak{f} $ and $\mathfrak{p}$. From Lemma \ref{ldes}, $\langle\cdot,\cdot\rangle$ takes the form
(\ref{inner}). Because $[\mathfrak{f} ,\mathfrak{f} ]_{\mathfrak{m} }\subset\mathfrak{f} $, equation (\ref{ccc})
implies that $[\mathfrak{f} ,\mathfrak{f} ]_{\mathfrak{m} } = 0$, that is, $[\mathfrak{f} ,\mathfrak{f} ]\subset
\mathfrak{k}$. Hence $(\mathfrak{l},s)$ is an orthogonal symmetric Lie algebra,
where $s$ is the involutive automorphism of $\mathfrak{l}$ with eigenspaces $\mathfrak{k}$ and $\mathfrak{f}$ for the eigenvalues $+1$ and
$-1$, respectively and, since $L$ is connected, $(L,K)$ is a pair associated to it.

Conversely, if
$(\mathfrak{l},s)$ is an orthogonal symmetric Lie algebra associated to $(L,K)$, 
one has $[\mathfrak{f} ,\mathfrak{f} ]\subset\mathfrak{k}$
and hence $[\mathfrak{f} ,\mathfrak{f} ]_{\mathfrak{m} }=0$. So if $\langle\cdot,\cdot\rangle$ is an
inner product on $\mathfrak{m} $ we have $\mathfrak{S}_{XYZ}\langle[X,Y]_\mathfrak{m} ,Z\rangle=0$
for every $X,Y,Z\in \mathfrak{f} $.

Moreover, since the subspaces $\mathfrak{f} $ and $\mathfrak{p}$ of $\mathfrak{m} $ are invariant by the adjoint representation 
  $\mathrm{Ad}\colon K \rightarrow \mathrm{Aut}(\mathfrak{m} )$, the inner product $\langle\cdot,\cdot\rangle$
on $\mathfrak{m} $ making $\mathfrak{f} $ and $\mathfrak{p}$ orthogonal and given by
\[
  \langle\cdot,\cdot\rangle =-2\,B_{|\mathfrak{f} \times\mathfrak{f} } +
    \,B_{|\mathfrak{p}\times\mathfrak{p}}
  \]
 is $\mathrm{Ad}(K)$-invariant. Hence, we
also have $\mathfrak{S}_{XYZ}\langle[X,Y]_\mathfrak{m} ,Z\rangle=0$ if either $X,Y\in\mathfrak{f} $ and $Z\in\mathfrak{p}$ (because
$[\mathfrak{f} ,\mathfrak{f} ]_{\mathfrak{m} }=0$ and $[\mathfrak{f} ,\mathfrak{p}]\subset\mathfrak{p}$) or if $X,Y,Z\in \mathfrak{p}$
(as $[\mathfrak{p},\mathfrak{p}]_{\mathfrak{m} } \subset \mathfrak{f} $).

Finally, if $X,Y\in \mathfrak{p}$ and $Z\in \mathfrak{f} $, using that $B$ is $\mathrm{Ad}(G)$-invariant, one gets
  \begin{align*}
   \raisebox{-0.2ex}{\Large$\mathfrak{S}$}_{XYZ}\langle[X,Y]_\mathfrak{m} ,Z\rangle & =
-2B([X,Y]_\mathfrak{m} ,Z) + B([Y,Z]_\mathfrak{m} ,X)+B([Z,X]_\mathfrak{m} ,Y) \\
          & =-2B([X,Y],Z) + B([Y,Z],X)+B([Z,X],Y) = 0. \qedhere
   \end{align*}
\end{proof}
\begin{remark} \em
\label{simplyconnected} 
Since $G/L$ is contractible,
the exact homotopy sequence of the fibration~$\pi$,
\[
    \dotsb \rightarrow  \pi_2(G/L) \rightarrow  \pi_1(L/K)  \rightarrow   \pi_1(G/K) \rightarrow
    \pi_1(G/L),
\]
shows that
 $L/K$ is simply connected if and only if $M= G/K$ is so. 
\end{remark}
The fibre $F_{x}$ through an arbitrary point $x=\tau_{a}(o)$, for some $a\in G$, is given by $F_{x} = \tau_{a}(F)$, where $F$ is the fibre type $F = L/K$. With respect to an adapted metric, the fibres are totally geodesic but the corresponding foliation is not necessarily Riemannian (see \cite{Gon} for more details). From Proposition \ref{noncomsemsim}, the existence of adapted cyclic metrics on $M$ implies that $F$ is locally symmetric and, from Remark
\ref{simplyconnected}, it is globally symmetric if $M$ is simply connected.  Moreover, due to the connectedness of $K$, also $F$ is globally symmetric if the orthogonal symmetric Lie algebra $(\mathfrak{l},s)$ is of compact type, that is, $\mathfrak{l}$~is semisimple. 

In order to obtain Tables
$1$, $2$ and $3$, we shall need the following lemma.

\begin{lemma}
\label{lclassification}
Let $(\mathfrak{l} ,s)$ be an orthogonal symmetric Lie algebra with
$\mathfrak{l}$ a compact Lie algebra. Let $\mathfrak{k}$ and $\mathfrak{f} $ be
the eigenspaces of $s$ for the eigenvalues $+1$ and $-1$, respectively. Suppose that the dimension
of the center $\mathfrak{z}(\mathfrak{l} )$ of $\mathfrak{l} $ is less than or equal to  one. If the Lie
algebra $\mathrm{ad}_{\mathfrak{l} }(\mathfrak{k})$ acts irreducibly on $\mathfrak{f}$,
$(\mathfrak{l} ,s)$ belongs to one of the following classes:

\begin{enumerate}[{\rm(i)}]
\item {\rm Type S$1:$} $(\mathfrak{l} ,s)$ is an irreducible orthogonal symmetric Lie algebra$;$
\item {\rm Type S$2:$} $\mathfrak{l} $ is semisimple and there exists a proper ideal
$\tilde{\mathfrak{l} }$ of $\mathfrak{l} $ such that the pair $(\tilde{\mathfrak{l} },\tilde{s})$,
where $\tilde{s}$ is the restriction of $s$ to $\tilde{\mathfrak{l} }$, is an irreducible
orthogonal symmetric Lie algebra and $\mathfrak{f} $ is the $-1$-eigenspace of $\tilde{s};$
\item {\rm Type NS$0:$} $\mathfrak{l} $ is not semisimple and $(\mathfrak{l} ,s)$ is effective. Then $\mathfrak{f} =\mathfrak{z}(\mathfrak{l})$ and so, $(\mathfrak{l} ,s)$ is of Euclidean type$;$
\item {\rm Type NS$1:$} $\mathfrak{l} $ is not semisimple, $(\mathfrak{l} ,s)$ is not effective and
$({\mathfrak l}_{-},s_{-})$ is an irreducible orthogonal symmetric Lie algebra, where $\mathfrak{l} _{-}$
is the semisimple ideal $[\mathfrak{l} ,\mathfrak{l} ]$ of $\mathfrak{l} ;$
\item {\rm Type NS$2:$} $\mathfrak{l} $ is not semisimple, $(\mathfrak{l} ,s)$ is not effective and there
exists a proper ideal $\tilde{\mathfrak{l} }_{-}$ of $\mathfrak{l} _{-}$ such that the pair
$(\tilde{\mathfrak{l} }_{-},\tilde{s}_{-})$, where $\tilde{s}_{-}$ is the restriction of $s$ to
$\tilde{\mathfrak{l} }_{-}$, is an irreducible orthogonal symmetric Lie algebra and $\mathfrak{f} $
is the $-1$-eigenspace of $\tilde{s}_{-}$.
\end{enumerate}
\end{lemma}
\begin{proof} 
First, suppose that $\mathfrak{l} $ is semisimple. Then the pair $(\mathfrak{l} ,s)$
is an orthogonal symmetric Lie algebra of compact type. When $\mathfrak{k}$ contains no  nonzero
ideal of $\mathfrak{l} $, $(\mathfrak{l} ,s)$ is irreducible, and it is of type S$1$.
Otherwise, let $\mathfrak{u}$ be the maximal ideal of $\mathfrak{l} $ contained in $\mathfrak{k}$ and let
$\tilde{\mathfrak{l} }$ be the semisimple compact ideal, orthogonal complement of $\mathfrak{u}$ with respect
to $B$. Because $\mathfrak{u}\subset \mathfrak{k}$ and $B$ is invariant by $s$, it follows that $s$ preserves
$\tilde{\mathfrak{l} }$. Hence, the restriction $\tilde{s}$ of $s$ to $\tilde{\mathfrak{l} }$ is an involutive
automorphism of $\tilde{\mathfrak{l} }$. Since the set $\tilde{\mathfrak{k}} = \mathfrak{k}\cap \tilde{\mathfrak{l} }$
of fixed points of $\tilde{s}$ is a compactly embedded subalgebra of $\mathfrak{l} $, we have that
$(\tilde{\mathfrak{l} },\tilde{s})$ is an orthogonal symmetric Lie algebra, which is moreover by construction
irreducible and $\mathfrak{f} $ is the $-1$-eigenspace of $\tilde{s}$. Then, $(\mathfrak{l} ,s)$ is of type S$2$.

Next, suppose that $\mathfrak{l} $ is not semisimple. Then $\mathfrak{l} $ is the direct sum $\mathfrak{l}
=\mathfrak{l} _{0}\oplus \mathfrak{l} _{-}$, where $\mathfrak{l} _{0}$ and $\mathfrak{l} _{-}$ are the ideals
$\mathfrak{z}(\mathfrak{l} )$ and $[\mathfrak{l} ,\mathfrak{l} ]$, respectively. These subspaces are invariant
under $s$ and orthogonal with respect to the Killing form $B$ of $\mathfrak{l} $. Then $(\mathfrak{l} _{0},s_{0})$
and $(\mathfrak{l} _{-},s_{-})$ are orthogonal symmetric Lie algebras, where $s_{0}$ and $s_{-}$ denote the
restrictions of $s$ to $\mathfrak{l} _{0}$ and $\mathfrak{l} _{-}$, respectively \cite[Chapter  V, Theorem 1.1]{Hel}.
Let $\mathfrak{l} _{0} = \mathfrak{k}_{0}\oplus \mathfrak{f} _{0}$ and $\mathfrak{l} _{-} =
\mathfrak{k}_{-}\oplus\mathfrak{f} _{-}$ be the decomposition of $\mathfrak{l} _{0}$ and $\mathfrak{l} _{-}$
into the corresponding $\pm 1$-eigenspaces of $s_{0}$ and $s_{-}$. Then the subspaces $\mathfrak{k}_{0}$
and $\mathfrak{k}_{-}$ are ideals in $\mathfrak{k}$, orthogonal with respect to $B$, and $\mathfrak{k} =
{\mathfrak k}_{0}\oplus \mathfrak{k}_{-}$. Because the subspaces $\mathfrak{f} _{0}$ and $\mathfrak{f} _{-}$
of $\mathfrak{f} $ are $\mathrm{ad}_{\mathfrak{l} }(\mathfrak{k})$-invariant, one of them must be trivial. If
$(\mathfrak{l} ,s)$ is effective, that is, $(\mathfrak{l} ,s)$ is of type NS$0$, one gets that $\mathfrak{k}_{0}
=\{0\}$ and $\mathfrak{l} _{0} = \mathfrak{f} _{0}$. Hence, $\mathfrak{f} _{-}$ is trivial and so,
$\mathfrak{l} _{0} = \mathfrak{f} $ and $\dim {\mathfrak l}_{0} = 1$.

If $(\mathfrak{l} ,s)$ is not effective, then $\mathfrak{l} _{0} = \mathfrak{k}_{0}$ and so, $\mathfrak{f} _{0}
 = \{0\}$. Hence, the types NS$1$ and NS$2$ for $(\mathfrak{l} ,s)$ correspond with the types
S$1$ and S$2$ for the orthogonal symmetric Lie algebra $(\mathfrak{l} _{-},s_{-})$
(of compact type), respectively. Finally, the case (v), for the type NS2, is proved using the same proof as for the case (ii).
\end{proof}

\smallskip

Given an irreducible Riemannian symmetric space of noncompact type $G/L$, where $L$ is a maximal compact subgroup of $G$, we consider orthogonal symmetric Lie algebras $(\mathfrak{l},s)$, where $\mathfrak{l}$ is the Lie algebra of $L$, satisfying the following two conditions:
\begin{enumerate}
\item[{\rm (i)}] the algebra $\mathrm{ad}_{\mathfrak{l} }(\mathfrak{k})$ acts irreducibly on $\mathfrak{f}$, where $\mathfrak{k}$ and $\mathfrak{f}$ are the eigenspaces of $s$ for the eigenvalues $+1$ and $-1$, respectively;
\item[{\rm (ii)}] the connected subgroup $K$ of $L$, with Lie algebra $\mathfrak{k}$, is closed (this is the case if $\mathfrak{l}$ is semisimple).
\end{enumerate}
Then, from Proposition \ref{noncomsemsim}, the homogeneous manifold $G/K$ admits a cyclic metric adapted to the transversally symmetric fibration $\pi\colon G/K\to G/L$ and the fibre type $F = L/K$ is isotropy-irreducible. Since $(\mathfrak{l},s)$ satisfies the hypothesis of Lemma \ref{lclassification}, we can give the next
\begin{definition}
\label{deftype}
We say that the cyclic homogeneous Riemannian manifold $G/K$ has {\it irreducible fibre of type} S1, S2, NS0, NS1 {\it or\/} NS2 if $(\mathfrak{l},s)$ 
is of that type.
\end{definition}

All these cyclic homogeneous Riemannian manifolds $G/K$ with irre\-du\-ci\-ble fibre are listed as the pairs $(G,K)$ on the first and third columns of Tables $1$ and $2$ or the corresponding pairs of Lie algebras $(\mathfrak{g}, \mathfrak{k})$ in Table $3$. Note that all the connected subgroups $K$ in Tables $1$ and $2$ are compact. 

As for the simply connectedness of the manifolds $L/K$ (equivalenty, $G/K$) in Table $1$ for the non-semisimple Lie groups $L=U(n)$ and $L=S(U(p)\times U(q))$ when $L/K$ is not a circle, we give the next

\begin{remark} 
Consider the universal covering maps $\widetilde{\pi}\colon\widetilde{L}=\mathbb{R}\times SU(n)\rightarrow L=U(n)$ and $\widetilde{\rho}\colon \widetilde{L}\rightarrow U(1)\times SU(n)$, given by
$\widetilde{\pi}(t,A)=\mathrm{e}^{\mathrm{i} t}A$ and $\widetilde{\rho}(t,A)=(\mathrm{e}^{\mathrm{i} t},A)$, and the covering map $\rho\colon (e^{\mathrm{i} t},A)\in U(1)\times SU(n) \rightarrow \mathrm{e}^{\mathrm{i} t}A\in U(n)$, so that $\rho\makebox[7pt]{\raisebox{1.5pt}{\scriptsize $\circ$}}\widetilde{\rho}=\widetilde{\pi}$.
The  subgroup of  $U(n)$ defined by the maximal Lie subalgebra $\mathfrak{u}(1)\oplus\mathfrak{so}(n)$ of $\mathfrak{u}(1)\oplus\mathfrak{su}(n)$ is the compact group $K=\rho(U(1)\times SO(n))$. Now,
\[
\mathrm{Ker}\, \widetilde{\pi}_{|\mathbb{R}\times SO(n)}=\begin{cases}
                    \{(l\pi,\mathrm{e}^{-\mathrm{i} l\pi}) : l\in \mathbb{Z}\} & \textrm{ if } n \textrm{ even} \\
                     \{(2l\pi,\mathrm{e}^{-\mathrm{i} 2 l\pi}) : l\in \mathbb{Z}\} & \textrm{ if } n \textrm{ odd},
                     \end{cases}
                     \]
then
\[
\mathrm{Ker}\, \rho_{|U(1)\times SO(n)} =
 \widetilde{\rho}(\mathrm{Ker}\, \widetilde{\pi}_{|\mathbb{R}\times SO(n)})=
 \begin{cases}
      \{(1,I_n),(-1,-I_n)\} & \textrm{ if }  n \textrm{  even}\\
      \{(1,I_n)\}        & \textrm{ if }  n \textrm{ odd},
      \end{cases}
 \]
that is,
\[
            K=\rho(U(1)\times SO(n))=
            \begin{cases}
                        (U(1)\times SO(n))/\mathbb{Z}_2  & \textrm{ if }  n \textrm{ even}\\
                        U(1)\times SO(n)  & \textrm{ if } n \textrm{ odd}.
            \end{cases}
   \]
   If $\widetilde{K}=\widetilde{\pi}^{-1}(K)$, then $L/K$ and $\widetilde{L}/\widetilde{K}$ are diffeomorphic, and since $\widetilde L$ is connected and simply connected, the cardinal of $\pi_1(\widetilde{L}/\widetilde{K})$ coincides with the number of connected components of $\widetilde{K}$. Now, it is easy to see that
\[
  \widetilde{K}= \big(\mathbb{R}\times SO(n)\big) \cup \big(\mathbb{R}\times \mathrm{e}^{\mathrm{i} 2\pi/n} SO(n)\big) \cup \dots \cup \big(\mathbb{R}\times \mathrm{e}^{\mathrm{i} 2(n-1)\pi/n} SO(n)\big), 
\]
and $\widetilde{K}$ has either $n$ connected components (if $n$ is odd)  or  $n/2$ connected components (if $n$ is even). In particular, $L/K=U(n)/\rho(U(1)\times SO(n))$ is simply connected if and only if $n=2$.
Analogously, one obtains that $U(n)/((U(1)\times Sp(n/2)/\mathbb{Z}_2)$ is not simply connected for $n$ even $\geqslant 4$. On the other hand, the complex Grassmannian  $U(n)/(U(i)\times U(j))\approx SU(n)/S(U(i)\times U(j))$ is simply connected, and it is a homogeneous spin Riemannian symmetric spin space for $n=i+j$ even~\cite{{CahGut}}.

   In a similar way, if $L=S(U(p)\times U(q))$, $p\geqslant q\geqslant 1$, we consider its universal cover $\widetilde{L}=\mathbb{R}\times SU(p)\times SU(q)$ under the map  $\widetilde{\pi}\colon \widetilde{L}\rightarrow L$ given by
   $\widetilde{\pi}(t,A,B)=\mathrm{diag}(\mathrm{e}^{\mathrm{i} qt} A,\mathrm{e}^{-\mathrm{i} pt} B)$, which defines the covering map
  \[
  \rho\colon (\mathrm{e}^{\mathrm{i} t},A,B)\in U(1)\times SU(p)\times SU(q)\rightarrow
        \left(
\begin{array}{cc}
\mathrm{e}^{\mathrm{i} qt} A & 0 \\
\noalign{\smallskip}
0 & \mathrm{e}^{-\mathrm{i} pt} B
\end{array}
\right )
  \in L,
  \]
  and this can be used to obtain the (compact) subgroups, among others, of $L$ defined by the maximal compact subalgebras
   $\mathfrak{u}(1)\oplus \mathfrak{su}(p)\oplus \mathfrak{so}(q)$, 
   $\mathfrak{u}(1)\oplus \mathfrak{su}(p)\oplus \mathfrak{sp}(q/2)$ (for $q$ even) 
  and $\mathfrak{u}(1)\oplus\Delta\mathfrak{su}(p)$ (for $p=q$)
  of $\mathfrak{u}(1)\oplus \mathfrak{su}(p)\oplus \mathfrak{su}(q)$.
   We can also see via the nonconnectedness of $\widetilde{K}=\widetilde{\pi}^{-1}(K)$ that if 
   $K=\rho(U(1)\times SU(p)\times SO(q))$ (for $q\geqslant 3$),
   or $\rho(U(1)\times SU(p)\times Sp(q/2))$ (for $q$ even $\geqslant 4$),
  or $\rho(U(1)\times \Delta SU(p))$ (for $p=q\geqslant 3$),  then $L/K$ is not simply connected. Moreover, if $K=S(U(i)\times U(j)\times U(k))$ (for $i+j=p$, $k=q$, or $i=p$, $j+k=q$), then $\widetilde{\pi}^{-1}(K)$  is the connected subgroup $\mathbb{R}\times S(U(i)\times U(j))\times SU(q)$
 or  $\mathbb{R}\times SU(p)\times S(U(j)\times U(k))$ of $\widetilde{L}$ and so in this case $L/K$ is simply connected. 
\end{remark}

\section{Homogeneous spin structures on transversally symmetric fibrations of noncompact type}
\label{secfiv}

Let $\xi = (E,\pi,B;F)$ be a fibration and let $\tau_F$ denote the bundle along the fibres of $\xi$.
The quotient of the tangent bundle $T(E)$
of the total space $E$ of $\xi$ by $\tau_F$ is equivalent \cite[Proposition  7.6]{BorHir} to the
fibration induced by $\pi$ from the tangent bundle of the base space $B$, so $T(E)$ decomposes as the Whitney sum
\begin{equation}
\label{tetfptb}
T(E) = \tau_F \oplus \pi^*T(B)
\end{equation}
and hence the corresponding total Stiefel-Whitney classes satisfy $w(T(E)) = w(\tau_F)$ $w(\pi^*T(B))$. Thus,
\begin{equation}
\label{w1w2}
\begin{split}
& \hspace{17mm} w_1(T(E)) = w_1(\tau_F) + w_1(\pi^*T(B)) , \\
& w_2(T(E)) = w_2(\tau_F) + w_1(\tau_F) w_1(\pi^*T(B)) + w_2(\pi^*T(B)) .
\end{split}
\end{equation}

In our case, the homogeneous fibration $L/K \overset{i}{\hookrightarrow} G/K \overset{\pi}{\to} G/L$ induces the fibration 
$T(L/K) \overset{i_*}{\hookrightarrow} T(G/K) \overset{\pi_*}{\to} T(G/L)$ and we have the following result.

\smallskip

\begin{proposition}
\label{conspi}
Let $\pi\colon M = G/K \to N = G/L$ be a transversally symmetric fibration of noncompact type. Then we have:
\begin{enumerate}
\item[{\rm (i)}] $M$ is a spin manifold if and only if the fibre type $F=L/K$ is so$;$
\item[{\rm (ii)}] if $F$ is simply connected then $M$ is a homogeneous spin Riemannian manifold if and only if $F$ is
also a homogeneous spin Riemannian manifold.
\end{enumerate}
\end{proposition}

\begin{proof} 
(i) We have $w_1(G/L)=0$ and $w_2(G/L)=0$. Furthermore, formula \eqref{tetfptb} is here
\[
T(G/K) =  i_*(T(L/K)) \oplus \pi^*(T(G/L))
\]
and $w_1\big(\pi^*T(G/L)\big)=0$, $w_2\big(\pi^*T(G/L)\big)=0$. It follows from \eqref{w1w2} that
$w_1(G/K)$ $=0$ if and only if $w_1(L/K) = 0$, and $w_2(G/K)=0$ if and only if $w_2(L/K) = 0$.

\smallskip

(ii) follows by using Remark \ref{simplyconnected}, Proposition \ref{psimple} and (i).
\end{proof}

\smallskip

Next, we show Theorem \ref{conspi1}, which generalizes Proposition \ref{conspi},(ii) when the fibre type $F$ is not necessarily simply connected.

\medskip

\noindent {\sf Proof of Theorem \ref{conspi1}.} From Proposition \ref{psimple} and Remark \ref{pnon-comp}, $N$ is a homogeneous spin Riemannian manifold, $\widetilde{G}/\widetilde{L}$ is an adapted quotient expression and $\mathfrak{g} = \mathfrak{l}\oplus \mathfrak{p}$ is an adapted reductive decomposition. 

Suppose that $(M = G/K,g)$ is a homogeneous spin Riemannian manifold, where $g$ is a metric adapted to the homogeneous fibration and $\widetilde{G}/\widetilde{K}$ and $\mathfrak{g} = \mathfrak{k}\oplus \mathfrak{m}$ are an adapted quotient expression and an adapted reductive decomposition for the homogeneous spin structure. Since $\mathfrak{m}$ and $\mathfrak{p}$ are $\mathrm{Ad}^{\widetilde{G}}(\widetilde{K})$-invariant subspaces of $\mathfrak{g}$ and $\mathfrak{f}$ and $\mathfrak{p}$ are $B$-orthogonal in $\mathfrak{m}$, we have that $\mathfrak{f}$ must also be $\mathrm{Ad}^{\widetilde{G}}(\widetilde{K})$-invariant. Moreover, the restriction $\langle\cdot,\cdot\rangle_{\mathfrak{f}}$ to $\mathfrak{f}$ of the $\mathrm{Ad}^{\widetilde{G}}(\widetilde{K})$-invariant inner product $\langle\cdot,\cdot\rangle$ on $\mathfrak{m}$ which gives the adapted metric $g$ on $M$, is also $\mathrm{Ad}^{\widetilde{G}}(\widetilde{K})$-invariant. Applying Lemma \ref{luniversal}, $F$ can be expressed as the homogeneous Riemannian manifold $\widetilde{L}/\widetilde{K}$.  From the exact homotopy sequence of the fibration $\widetilde{L} \to \widetilde{G} \to  N=\widetilde{G}/\widetilde{L}$,
\[
  \dotsb \to \pi_2(N) \to \pi_1(\widetilde{L}) \to \pi_1(\widetilde{G}) \to
      \pi_1(N) \to \pi_0(\widetilde{L}) \to \pi_0(\widetilde{G}),
\]
one has that $\widetilde{L}$ is connected and simply connected. Then, since from Proposition \ref{conspi}(i), $F$ admits a spin structure, the result follows from Lemma \ref{lBar}.

Next, we shall prove the converse. Because $\mathfrak{f}$ and $\mathfrak{p}$ are $\mathrm{Ad}^{\widetilde{G}}(\widetilde{K})$-invariant subspaces of $\mathfrak{g}$, one gets that $\mathfrak{m} = \mathfrak{f}\oplus \mathfrak{p}$ is an  $\mathrm{Ad}^{\widetilde{G}}(\widetilde{K})$-invariant decomposition of $\mathfrak{m}$. We consider on $\mathfrak{m}$ the inner product $\langle\cdot,\cdot\rangle$ such that $\langle \mathfrak{f},\mathfrak{p}\rangle = 0$. This inner product coincides on $\mathfrak{p}$ with the $\mathrm{Ad}^{\widetilde{G}}(\widetilde{L})$-invariant inner product $\langle\cdot,\cdot\rangle_{\mathfrak{p}}$, which determines the $\widetilde{G}$-invariant metric on $N$, and on $\mathfrak{f}$ with the $\mathrm{Ad}^{\widetilde{G}}(\widetilde{K})$-invariant inner product which determines the $\widetilde{L}$-invariant metric on $F$. Then $\langle\cdot,\cdot\rangle$ is $\mathrm{Ad}^{\widetilde{G}}(\widetilde{K})$-invariant and Lemma \ref{luniversal} implies that $(M,g)$, $g$ being the adapted $G$-invariant metric determined by $\langle\cdot,\cdot\rangle$ on $\mathfrak{m}$, is isometric to $(\widetilde{G}/\widetilde{K},\tilde{g})$.

From Proposition \ref{conspi},(i), $M$ is a spin manifold. Then, using Lemma \ref{lBar}, $(M,g)$ is in fact a homogeneous spin Riemannian manifold with adapted quotient expression $\widetilde{G}/\widetilde{K}$  and reductive decomposition $\mathfrak{g} = \mathfrak{k}\oplus \mathfrak{m}$. 
The last part of the theorem is a direct consequence of Theorem \ref{dirac} and Proposition \ref{noncomsemsim}. \hfill $\square$  

\smallskip
 
Applying Theorem \ref{conspi1} for $F = S^{1}$, we have
\begin{corollary} 
The total space $M = G/K$ of a transversally symmetric fibration of noncompact type and with irreducible symmetric fibre of type {\rm NS0} is a homogeneous spin Riemannian manifold with the simplest Dirac operator.
\end{corollary}

\section{Example: The nonsymmetric traceless cyclic homogeneous spin Riemannian manifold $\widetilde{SL(2,\mathbb{R})}$}
\label{secsix}

The universal covering $\widetilde{SL(2,\mathbb{R})}$ of the Lie group $SL(2,\mathbb{R})$
has nonsymmetric related pair $(\mathfrak{g},\mathfrak{k}) = (\mathfrak{sl}(2,\mathbb{R}),0)$.
This corresponds to the fifth manifold in Table $1$, for $p=q=1$, via the isomorphism between
$\mathfrak{sl}(2,\mathbb{R})$ and $\mathfrak{su}(1,1)$, the latter being isomorphic to the nonsymmetric pair
$(\mathfrak{su}(1,1), \mathfrak{su}(1) \oplus \mathfrak{su}(1)) = (\mathfrak{su}(1,1),0)$.

Consider the basis of $\mathfrak{sl}(2,\mathbb{R})$ given by
\[
X_1 = a
\Bigl( \begin{array}{cc} 0  & \;\, 1 \\
                           1  & \;\, 0 \end{array} \Bigr), \quad
X_2 =b
\Bigl( \begin{array}{cc} 1  &  0 \\
                           0  &  -1 \end{array} \Bigr),\quad
X_3 = \frac{ab}{\sqrt{a^2+b^2}}
\Bigl( \begin{array}{cc} 0  & \;\,1 \\
                           -1 &  \;\,0 \end{array} \Bigr),
\]
for $a,b >0$, which is orthogonal with respect to the Killing form of $\mathfrak{sl}(2,\mathbb{R})$ and satisfies
\begin{eqnarray*}
[X_2,X_3] = r X_1, \quad [X_3,X_1] = s X_2, \quad [X_1,X_2] = t X_3, \quad
\end{eqnarray*}
where $r = 2b^2/\sqrt{a^2+b^2}$, $s = 2a^2/\sqrt{a^2+b^2}$, $t = -2\sqrt{a^2+b^2}$, so that $r,s>0$ and $r+ s + t= 0$.

Hence each left-invariant Riemannian metric $\langle\cdot,\cdot\rangle_{a,b}$
on $\widetilde{SL(2,\mathbb{R})}$, corresponding to the invariant inner product
on $\mathfrak{sl}(2,\mathbb{R})$ making $\{X_1,X_2,X_3\}$ orthonormal for such a given pair $a=\frac12\sqrt{t(r+t)} >0$, $b=\frac12\sqrt{t(s+t)} > 0$, 
is cyclic. Moreover, it is traceless, as
$\sum_{j=1}^3 \langle [X_j,X_i], X_j \rangle_{a,b} = 0$, $i=1,2,3$.

The manifold $\widetilde{SL(2,\mathbb{R})}$ is a Lie group, hence parallelizable, so
$(\widetilde{SL(2,\mathbb{R})},\langle\cdot,\cdot\rangle_{a,b})$ is a spin Riemannian manifold, for each couple $(a,b)$. Moreover,
since it is a traceless cyclic homogeneous Riemannian manifold, it
has Dirac operator like that on a Riemannian symmetric space. We now check this with more detail.

Denote by $\bar X_i$ the left-invariant vector fields on $\widetilde{SL(2,\mathbb{R})}$ defined by the elements $X_i$.
The Koszul formula for left-invariant vector fields $\bar X,\bar Y,\bar Z$,
\[
2 \langle \nabla_{\bar X}\bar Y,\bar Z \rangle = \langle [\bar X,\bar Y],\bar Z \rangle -
\langle [\bar Y,\bar Z],\bar X \rangle + \langle [\bar Z,\bar X],\bar Y \rangle,
\]
gives, applying the property $r + s + t = 0$, the components $\nabla_{\bar X_i}\bar X_j$ of the Levi-Civita connection of
$\langle\cdot,\cdot\rangle_{a,b}$,
\begin{alignat*}{3}
 & \nabla_{\bar X_1}\bar X_1 = 0, &\quad& \nabla_{\bar X_1}\bar X_2 = -r \bar X_3, &\quad& \nabla_{\bar X_1}\bar X_3 = r \bar X_2,\\
 \noalign{\smallskip}
 & \nabla_{\bar X_2}\bar X_1 = s \bar X_3, && \nabla_{\bar X_2}\bar X_2 = 0, && \nabla_{\bar X_2}\bar X_3 = -s \bar X_1, \\
 \noalign{\smallskip}
 & \nabla_{\bar X_3}\bar X_1 = -t \bar X_2, && \nabla_{\bar X_3}\bar X_2 = t \bar X_1,  && \nabla_{\bar X_3}\bar X_3 = 0.
\end{alignat*}

The fibre of $\Sigma(\widetilde{SL(2,\mathbb{R})})$ is
$\mathrm{Spin}(3)\cong SU(2)$. We have endomorphisms $\rho(X_i) \in \mathrm{End}(\Delta(\mathfrak{sl}(2,\mathbb{R})))
= \mathrm{End}(\mathbb{C}^2)$, $i=1,2,3$, and there is a well-defined Clifford multiplication
\[
\mathfrak{sl}(2,\mathbb{R})
\otimes \Sigma(\widetilde{SL(2,\mathbb{R})})_w \to \Sigma(\widetilde{SL(2,\mathbb{R})})_w,
\quad X \otimes \psi \mapsto X \cdot \psi = \rho(X)\psi,
\]
for all $w \in \widetilde{SL(2,\mathbb{R})}$ and $X\in \mathfrak{sl}(2,\mathbb{R})$. Since we have the relation 
$U \cdot V + V \cdot U = -2\langle U,V\rangle I$, 
for all $U,V \in \mathrm{Cl}_{\mathbb{C}}(\mathfrak{sl}(2,\mathbb{R}))$, and where $I$ denotes the identity element,
we can choose a basis $\{\psi_1,\psi_2\}$ of $\Delta(\mathfrak{sl}(2,\mathbb{R}))$ with respect to which one has
\[
\rho(X_1) =
\Bigl( \begin{array}{cc} \mathrm{i} & \;0\\
                                    0 & \;-\mathrm{i} \end{array} \Bigr), \quad
\rho(X_2) =
\Bigl( \begin{array}{cc} 0 & \;\,\mathrm{i} \\
                           \mathrm{i} & \;\,0 \end{array} \Bigr),\quad
\rho(X_3) =
\Bigl( \begin{array}{cc} 0 & \;-1 \\
                           1 & \;0 \end{array} \Bigr).
\]

The spin connection $\nabla^{\Sigma}$ induced by the Levi-Civita connection $\nabla$ in the 
fibration $\Sigma(\widetilde{SL(2,\mathbb{R})})$ is given, from \eqref{nablasig}, by
\[
\nabla^{\Sigma}_{\bar X_i} \psi = \bar X_i(\psi) + \frac14 \sum_{j,k=1}^3 \Gamma^k_{ij}
\bar X_j \cdot \bar X_k \cdot \psi, \qquad i=1,2,3,
\]
that is, by
\begin{gather*}
\nabla^{\Sigma}_{\bar X_1} \psi  = \bar X_1(\psi)
        - \frac{r}2\Bigl( \begin{array}{cc} \mathrm{i} & 0 \\
                                            0 &  - \mathrm{i}
                          \end{array} \Bigr)\psi, \quad
\nabla^{\Sigma}_{\bar X_2} \psi = \bar X_2(\psi)
        - \frac{s}2 \Bigl( \begin{array}{cc} 0 & \;\,\mathrm{i} \\
                                             \mathrm{i} & \;\,0
                          \end{array} \Bigr)\psi, \\
\nabla^{\Sigma}_{\bar X_3} \psi  = \bar X_3(\psi)
       - \frac{t}2 \Bigl( \begin{array}{cc} 0 & -1 \\
                                    1 & 0
                          \end{array} \Bigr)\psi,
\end{gather*}
so, from  \eqref{dir}, the Dirac operator is given by
\[
D\psi = \sum_{i=1}^3 X_i \cdot X_i(\psi)
- \frac12 ( r + s + t ) \Bigl( \begin{array}{cc} -1 & 0 \\
                                    0 & -1
                          \end{array} \Bigr)\psi
         = \sum_{i=1}^3 X_i \cdot X_i(\psi),
\]
as expected.

{\footnotesize
\begin{table}[htb]
\caption{\footnotesize Nonsymmetric cyclic homogeneous Riemannian manifolds $G/K$ of a (necessarily noncompact)
classical simple absolutely real Lie group $G$, where $K$ is a closed connected subgroup of the maximal compact subgroup $L$ of $G$ and $L/K$ is isotropy-irreducible. The manifolds $G/K$ marked with
$\star$ on the seventh column, satisfying moreover the conditions on the sixth
column, if any, are examples of nonsymmetric homogeneous spin Riemannian manifolds with Dirac operator like that on a Riemannian symmetric spin space. (As for the type, see Definition \ref{deftype}.)} 
\label{table1}
\setlength\arraycolsep{2pt}
\[
\begin{tabular}{llllllc}  \hline\noalign{\smallskip}
$G$                 & $L$                        & $K$                                                &           & Type                         & & Spin \\ \noalign{\smallskip}\hline\noalign{\smallskip}
$SL(n,\mathbb{R})$  & $SO(n)$                    & $SO(i) \times SO(j)$                & (a)  &  S1       & $j=1$  & $\star$ \\
\quad $(n\geqslant 3)$ &&&&&or $n$ even& \\
   &                     & $U(n/2)$                           & (b)  &  S1                 & $n\geqslant 6$ & $\star$ \\ \hline
$SU^*(2n)$          & $Sp(n)$                    & $U(n)$                                             &           &  S1                & $n\; {\rm odd}\geqslant 3$  & $\star$ \\
   \quad $(n\geqslant 2)$               &                          & $Sp(i) \times Sp(j)$           & (a)  &  S1                  &   & $\star$ \\ \hline
$SU(p,q)$           & $S(U (p)$      & $SU(p) \times SU(q)$             &           &  NS0                      &   & $\star$ \\
  \quad $(p\geqslant q\geqslant 1)$   & \;\;$ \times U(q))$ & $S(U (i)\times U (j) \times U (k))$                & (c) &  NS2  &   & \\
                  &                          & $(U(1)\times SO(p) \times  SU(q))/\mathbb{Z}_q$    & (d) &   NS2                &  & \\
                  &                          &  $(U(1)\times SO(p) \times SU(q))/\mathbb{Z}_{2q}$  & (e) &   NS2       &  & \\
                  &                          & $(U(1)\times SU(p) \times SO(q))/\mathbb{Z}_p$     & (f) &    NS2               &  & \\
                  &                          & $(U(1)\times SU(p) \times SO(q))/\mathbb{Z}_{2p}$  & (g) &    NS2               &  & \\
                  &                          & $(U(1)\times Sp(p/2)\times SU(q))/\mathbb{Z}_{2q}$ & (e) &    NS2               &  & \\
                  &                          & $(U(1)\times SU(p)\times Sp(q/2))/\mathbb{Z}_{2p}$ & (g) &    NS2              &  & \\
                  &                          & $U(1)\times \Delta SU(p)$                          & (h) &    NS1     &  & \\
                  &                          & $(U(1)\times \Delta SU(p))/\mathbb{Z}_2$           & (i) &  NS1         &  & \\ \hline
$SO_0(p, 1)$        & $SO(p)$                    & $SO(i)\times SO(j)$              &   (j) &  S1  & $j=1$   & $\star$ \\
\quad $(p\geqslant 3)$ &&&&&or $p$ even& \\
             &                       & $U(p/2)$                                      & (e) &  S1  &  $p\geqslant 6$ & $\star$ \\  \hline
$SO_0(p, 2)$        & $SO(p)$        &    $SO(p)$                                     &  &  NS0  &   &    $\star$    \\
  \quad $(p\geqslant 3)$& \;\;$\times SO(2)$        &   $SO(i)\times SO(j) \times SO(2)$  &  (j)      &  NS1 &  $j=1$   & $\star$          \\
 &&&&&or $p$ even& \\
                  &                          & $U(p/2)\times SO(2)$                              & (e) &  NS1  & $p\geqslant 6$  & $\star$ \\ \hline
$SO_0(p, q)$        & $SO(p)$       & $SO(i) \times SO(j) \times SO(q)$ & (j) &  S2   & $j=1$ or   & $\star$ \\
\quad $(p\geqslant q\geqslant 3)$       & \;\;$\times SO(q)$       &  & &    & $p$ even   &  \\
  &  & $SO(p) \times SO(j) \times SO(k)$ & (k) & S2 &  $k=1$ or & $\star$ \\
  &&&&& $q$ even & \\
                  &    & $U(p/2)\times SO(q)$                                   & (e) &  S2  & $p\geqslant 6$ & $\star$ \\
                  &        & $SO(p) \times U(q/2)$        & (g) &  S2 & $q\geqslant 6$    & $\star$ \\
                  &                          & $\Delta SO(p)$                                     & (l) &  S1  &    & $\star$ \\ \hline
$SO^*(2n)$          & $U(n)$                     & $U(1)\times  SO(n)$                                & (m) &  NS1     &    & \\
   \quad $(n\geqslant 3)$ &                  & $(U(1)\times  SO(n))/\mathbb{Z}_2$                   & (b) &  NS1      &    & \\
                  &                          & $U(i) \times U(j)$                                           & (a) &  NS1  & n \; {\rm even} & $\star$ \\
                  &                          & $SU(n)$                                            &          &  NS0  &  & $\star$ \\
                  &                          & $(U(1)\times Sp(n/2))/\mathbb{Z}_2$                & (b) &  NS1  &    & \\ \hline
$Sp(n,\mathbb{R})$  & $U(n)$                     & $U(1)\times  SO(n)$                                 & (m) &  NS1  &    & \\
   \quad $(n\geqslant 2)$ &                  & $(U(1)\times  SO(n))/\mathbb{Z}_2$                   & (b) &  NS1  &    & \\
                  &                          & $U(i) \times U(j)$                                           & (a) &  NS1  & n \;{\rm  even} & $\star$ \\
                  &                          & $SU(n)$                                            &          &  NS0  & & $\star$ \\
                  &                          & $(U(1)\times Sp(n/2))/\mathbb{Z}_2$                & (n) &  NS1  &  & \\ \hline
$Sp(p,q)$           & $Sp(p)$        & $U(p)\times Sp(q)$                    &          &  S2   & $p \; {\rm odd}\geqslant 3$   & $\star$ \\
  \quad $(p\geqslant q\geqslant 1)$     & \;\; $\times Sp(q)$  & $Sp(p) \times U(q)$                                &          &  S2   & $q \; {\rm odd}\geqslant 3$   & $\star$ \\
                  &                          & $Sp(i) \times Sp(j) \times Sp(k)$                  & (c) &  S2   &    & $\star$ \\
                  &                          & $\Delta Sp(p)$                           & (l)  &  S1   &    & $\star$ \\ \noalign{\smallskip}\hline
\end{tabular}
\]
{\rm(a)} $n=i+j$, $i\geqslant j\geqslant 1;$
{\rm(b)} $n$ {\rm even;}
{\rm(c)} {\rm either} $i+j=p$, $i\geqslant j \geqslant 1$, $k=q$ {\rm or} $i=p$, $j+k=q$, $j\geqslant k \geqslant 1;$
{\rm(d)} {\rm $p$ odd;}
{\rm(e)} $p$ {\rm even};
{\rm(f)} {\rm $q$ odd;}
{\rm(g)} $q$ {\rm even};
{\rm(h)} $p=q$ {\rm odd};
{\rm(i)} $p=q$ {\rm even};
{\rm(j)} $p=i+j$,  $i\geqslant j\geqslant 1 $;
{\rm (k)} $q=j+k$,  $j\geqslant k\geqslant 1 $;
{\rm(l)} $p=q$;
{\rm(m)} $n$ {\rm odd};
{\rm(n)} $n$ {\rm even} $\geqslant 4$.
\end{table}}

{\footnotesize
\begin{table}[htb]
\caption{\footnotesize Simply connected nonsymmetric cyclic  homogeneous Riemannian manifolds $G/K$ of a
simple complex Lie group $G$ (considered as a real Lie group), where $K$ is a closed connected subgroup of the maximal compact subgroup $L$ of $G$ and $L/K$ is isotropy-irreducible. {\it All of them are of type\/} S1 and are nonsymmetric homogeneous spin Riemannian manifolds
with Dirac operator like that on a Riemannian symmetric spin space,
whenever they satisfy moreover the necessary conditions on the fifth column, if any. The space marked with ``no'' is not spin.}
\label{table2}
\setlength\arraycolsep{4pt}
\[
\begin{tabular}{lllcl}  \hline\noalign{\smallskip}
$G$                   & $L$     & $K$                               &           & Spin \\ \noalign{\smallskip}\hline\noalign{\smallskip}
$SL(n,\mathbb{C})$    & $SU(n)$ & $SO(n)$                                        & (a)  & $n$ even   \\
                    &       &  $S(U(i)\times U(j))$                          & (b)  & $n$ even   \\
                    &       & $Sp(n/2)$                                      & (c)  &    \\ \hline
$SO(n,\mathbb{C})$    & $SO(n)$ & $SO(i)\times SO(j)$                            & (d)  &  j=1 \; {\rm or} \; n \; {\rm even}   \\
                    &       & $U(n/2)$                                       & (e)  &    \\ \hline
$Sp(n,\mathbb{C})$    & $Sp(n)$ & $U(n)$                                         & (a)  & $n \; {\rm odd}\geqslant 3$   \\
                    &       & $Sp(i)\times Sp(j)$                            & (f)  &    \\ \hline
$E_6^{\,\mathbb{C}}$  & $E_6$   & $Sp(4)/\mathbb{Z}_2$                           &           &     \\
                    &       & $(SU(6)\times SU(2))/Z_2$                      &           &     \\
                    &       & $(\mathrm{Spin}(10)\times SO(2))/\mathbb{Z}_4$ &           &     \\
                    &       & $F_4$                                          &           &     \\ \hline
$E_7^{\,\mathbb{C}}$  & $E_7$   & $SU(8)/\mathbb{Z}_2$                           &           &     \\
                    &       & $(\mathrm{Spin}(12)\times SU(2))/\mathbb{Z}_2$ &           &     \\
                    &       & $(E_6\times U(1))/\mathbb{Z}_3$                &           &    \\ \hline
$E_8^{\,\mathbb{C}}$  & $E_8$   & $SO'(16)$                                      & (g)  &     \\
                    &       & $(E_7\times SU(2))/\mathbb{Z}_2$               &           &     \\ \hline
$F_4^{\,\mathbb{C}}$  & $F_4$   & $(Sp(3)\times SU(2))/\mathbb{Z}_2$             &           &  \bf{no}   \\
                    &       & $\mathrm{Spin}(9)$                             &           &     \\ \hline
$G_2^{\,\mathbb{C}}$  & $G_2$   & $(SU(2) \times SU(2))/\mathbb{Z}_2$            &           &    \\ \noalign{\smallskip}\hline
\end{tabular}
\]
{\rm(a)} $n\geqslant 2;$
{\rm(b)} $i+j=n\geqslant 3$, $i\geqslant j\geqslant 1;$
{\rm(c)} $n$ {\rm even}, $n\geqslant 4;$
{\rm(d)} $i+j=n\geqslant 5$, $i\geqslant j\geqslant 1;$ \\
{\rm(e)} $n$ {\rm even}, $n\geqslant 6;$
{\rm(f)} $i+j=n\geqslant 2$, $i\geqslant j\geqslant 1$.
{\rm(g)} see \cite[\S\S 8.2.11]{Wol}.
\end{table}}

{\footnotesize
\begin{table}[htb]
\caption{\footnotesize Triples of Lie algebras $(\mathfrak{g},\mathfrak{l},\mathfrak{k})$ of Lie groups $(G,L,K)$ such that $G/K$ is a nonsymmetric cyclic homogeneous Riemannian manifold of a (necessarily noncompact)
exceptional simple absolutely real Lie group $G$, where $K$ is a closed connected subgroup of the maximal compact subgroup $L$ of $G$ and $L/K$ is isotropy-irreducible. Except for the manifolds marked with ``no'' on the sixth column, the rest of manifolds $G/K$ are examples of nonsymmetric homogeneous spin Riemannian manifolds with Dirac operator like that on a Riemannian symmetric spin space.} 
\label{table3}
\setlength\arraycolsep{4pt}
\[
\begin{tabular}{lllccc}  \hline\noalign{\smallskip}
$\mathfrak{g}$                   & $\mathfrak{l}$  & $\mathfrak{k}$    &          & Type & Spin \\ \noalign{\smallskip}\hline\noalign{\smallskip}
$\mathfrak{e}_{6(6)}$   & $\mathfrak{sp}(4)$                   & $\mathfrak{u}(4)$                                      &          & S1 & {\bf no} \\
                      &                         & $\mathfrak{sp}(i)\oplus \mathfrak{sp}(j)$                         & (a) & S1 & \\ \hline
$\mathfrak{e}_{6(2)}$   & $\mathfrak{su}(6) \oplus \mathfrak{su}(2)$  & $\mathfrak{su}(6) \oplus \mathfrak{u}(1)$       &          & S2 & \\
                      &                         & $\mathfrak{so}(6) \oplus \mathfrak{su}(2)$                        &          & S2 & \\
                      &  & $\mathfrak{s}(\mathfrak{u}(i) \oplus \mathfrak{u}(j)) \oplus \mathfrak{su}(2)$          & (b) & S2 & \\ \hline
$\mathfrak{e}_{6(-14)}$ & $\mathfrak{so}(10) \oplus \mathfrak{so}(2)$
                                & $\mathfrak{so}(i)\oplus \mathfrak{so}(j) \oplus \mathfrak{so}(2)$                 & (c) &  NS1 & \\
                      &                         & $\mathfrak{u}(5) \oplus \mathfrak{so}(2)$                         &          &   NS1 & \\
                      &                         & $\mathfrak{so}(10)$                                               &          &  NS0 & \\ \hline
$\mathfrak{e}_{6(-26)}$ & $\mathfrak{f}_4$                     & $\mathfrak{sp}(3) \oplus \mathfrak{su}(2)$             &          & S1 & \bf{no} \\
                      &                         & $\mathfrak{so}(9)$                                                &          & S1 & \\  \hline
$\mathfrak{e}_{7(7)}$   & $\mathfrak{su}(8)$    & $\mathfrak{so}(8)$                                                    &          & S1 & \\
                      &                         & $\mathfrak{sp}(4)$                                                &          & S1 & \\
                      &                         & $\mathfrak{s}(\mathfrak{u}(i)\oplus \mathfrak{u}(j))$             & (d) & S1 & \\ \hline
$\mathfrak{e}_{7(-5)}$  & $\mathfrak{so}(12) \oplus \mathfrak{su}(2)$    & $\mathfrak{so}(12) \oplus \mathfrak{so}(2)$  &          & S2 & \\
                      &                & $\mathfrak{so}(i)\oplus \mathfrak{so}(j) \oplus  \mathfrak{su}(2)$         & (e) & S2 & \\
                      &                         & $\mathfrak{u}(6) \oplus  \mathfrak{su}(2)$                        &          & S2 & \\ \hline
$\mathfrak{e}_{7(-25)}$ & $\mathfrak{e}_6 \oplus \mathfrak{so}(2)$       & $\mathfrak{sp}(4) \oplus \mathfrak{so}(2)$   &          &  NS1& \\
             &                         & $\mathfrak{su}(6) \oplus \mathfrak{su}(2)  \oplus \mathfrak{so}(2)$        &          &  NS1& \\
             &                         & $\mathfrak{so}(10) \oplus \mathfrak{so}(2)   \oplus \mathfrak{so}(2)$      &          &  NS1& \\
             &                         & $\mathfrak{f}_4   \oplus \mathfrak{so}(2)$                                &          &  NS1& \\
             &                         & $\mathfrak{e}_6$                                                           &          & NS0 & \\ \hline
$\mathfrak{e}_{8(8)}$   & $\mathfrak{so}(16)$     & $\mathfrak{so}(i) \oplus \mathfrak{so}(j)$                          & (f) & S1 & \\
             &                         & $\mathfrak{u}(8)$                                                          &          & S1 & \\ \hline
$\mathfrak{e}_{8(-24)}$ & $\mathfrak{e}_7 \oplus \mathfrak{su}(2)$        & $\mathfrak{e}_7 \oplus \mathfrak{so}(2)$    &          & S2 & \\
             &                         & $\mathfrak{su}(8)  \oplus \mathfrak{su}(2)$                                &          & S2 & \\
             &                         & $\mathfrak{so}(12)\oplus \mathfrak{su}(2) \oplus \mathfrak{su}(2)$         &          & S2 & \\
             &                         & $\mathfrak{e}_6 \oplus \mathfrak{so}(2) \oplus \mathfrak{su}(2)$           &          & S2 & \\ \hline
$\mathfrak{f} _{4(4)}$   &  $\mathfrak{sp}(3) \oplus \mathfrak{sp}(1)$ & $\mathfrak{sp}(3) \oplus \mathfrak{so}(2)$      &          & S2 & \\
                      &                         & $\mathfrak{sp}(2)\oplus \mathfrak{sp}(1) \oplus \mathfrak{sp}(1)$ &          & S2 & \\
                      &                         & $\mathfrak{u}(3) \oplus \mathfrak{sp}(1)$                         &          & S2 & \\ \hline
$\mathfrak{f} _{4(-20)}$ & $\mathfrak{so}(9)$                  & $\mathfrak{so}(i)\oplus \mathfrak{so}(j)$               & (g) & S1 & {\bf no} \\
                      &                         & $\mathfrak{so}(8)$                                                &          & S1 & \\ \hline
$\mathfrak{g}_{2(2)}$   & $\mathfrak{so}(4)$                   & $\mathfrak{u}(2)$                                               &          & S1 & \\
             &                         & $\mathfrak{so}(3)$                                                         &          & S1 & \\
             &                         & $\mathfrak{so}(2) \oplus \mathfrak{so}(2)$                                 &          & S1 & \\ 
\noalign{\smallskip}\hline
\end{tabular}
\]
{\rm(a)} $i+j=4$, $i\geqslant j\geqslant 1;$
{\rm(b)} $i+j=6$, $i\geqslant j \geqslant 1;$
{\rm(c)} $i+j=10$,  $i\geqslant j \geqslant  1;$
{\rm(d)} $i+j=8$, $i\geqslant j \geqslant 1;$ 
{\rm(e)} $i+j=12$,  $i\geqslant j \geqslant 1$
{\rm(f)} $i+j=16$,  $i\geqslant j \geqslant 1;$
{\rm(g)} $i+j=9$,  $i\geqslant j\geqslant 2$.
\end{table}}

{\footnotesize
\begin{table}
\caption{\footnotesize
Cahen-Gutt spaces (compact simply connected Riemannian symmetric spin spaces $L/K$ with $L$ simple).}
\label{table4}
\begin{center}
{\setlength\arraycolsep{0.2pt}
\begin{tabular}{@{}lll} \hline\noalign{\smallskip}
$L$          & $K$                                                    &                \\  \noalign{\smallskip}\hline\noalign{\smallskip}
$SU(n)$      & $SO(n)$                                                & $n$ \rm{even} $\geqslant 4$                                     \\  \hline
$SU(p+q)$    & $S(U(p) \times U(q))$                                  & $p+q$ \rm{even}, $p\geqslant q\geqslant 1$                  \\  \hline
$SU(2n)$     & $Sp(n)$                                                & $n \geqslant 2$                                    \\  \hline
$SO(n)$      & $SO(n-1)$                                              & $n\geqslant 3$                                     \\  \hline
$SO(p+q)$    & $SO(p) \times SO(q)$                                   & $p+q$ \rm{even}, $p\geqslant q\geqslant 2$ \\  \hline
$SO(2n)$     & $U(n)$                                                 & $n \geqslant 3$                                    \\  \hline
$Sp(n)$      & $U(n)$                                                 & $n$ \rm{odd}  $\geqslant 3$                      \\  \hline
$Sp(p+q)$    & $Sp(p)\times Sp(q)$                                    & $p\geqslant q\geqslant 1$                                   \\  \hline
$E_6$        & $Sp(4)/\mathbb{Z}_2$                                   &                                                    \\
             & $(SU(6)\times SU(2))/\mathbb{Z}_2$                     &                                                    \\
             & $(\mathrm{Spin}(10)\times SO(2))/\mathbb{Z}_4$         &                                                    \\
             & $F_4$                                                  &                                                    \\  \hline
$E_7$        & $SU(8)/\mathbb{Z}_2$                                   &                                                    \\
             & $(\mathrm{Spin}(12) \times SU(2))/\mathbb{Z}_2$        &                                                    \\
             & $(E_6 \times U(1))/\mathbb{Z}_3$                       &                                                    \\  \hline
$E_8$        & $SO'(16)$                                              &                                                    \\
             & $(E_7 \times SU(2))/\mathbb{Z}_2$                      &                                                    \\  \hline
$F_4$        & $\mathrm{Spin}(9)$                                     &                                                    \\  \hline
$G_2$        & $(SU(2) \times SU(2))/\mathbb{Z}_2$                    &                                                    \\  \hline
\end{tabular}
}
\end{center}
\end{table}}


\begin{thebibliography}{99}

\bibitem{AmmBar} B.~Ammann and C.~B\"ar, The Dirac operator on nilmanifolds and collapsing circle bundles, Ann.\ Global
Anal.\ Geom.\ {\bf 16} (1998), 221–-253. 

\bibitem{Bar} C.~B\"ar, The Dirac operator on homogeneous spaces and its spectrum on 3-dimensional lens spaces,
Arch.\ Math.\ (Basel) {\bf 59} (1992), 65–-79. 

\bibitem{BorHir} A.~Borel and F.~Hirzebruch, Characteristic classes and homogeneous spaces I, Amer.\ J. Math.\ {\bf 80} (1958), 458–-538. 

\bibitem{CahGut} M.~Cahen and S.~Gutt, Spin structures on compact simply connected Riemannian symmetric spaces,
Simon Stevin {\bf 62} (1988), 209--242. 

\bibitem{GadGonOub1} P.~M.~Gadea, J.~C.~G\'onzalez-D\'avila and 
J.~A.~Oubi\~na, Cyclic metric Lie groups, Monatsh.\ Math.\ {\bf 176}
(2015), 219--239. 

\bibitem{GadGonOub2} P.~M.~Gadea, J.~C.~G\'onzalez-D\'avila and 
J.~A.~Oubi\~na, Cyclic homogeneous Riemannian manifolds, Ann.\
Mat.\ Pura Appl.\ (4) (2015), doi: 10.1007/s10231-015-0534-7.

\bibitem{GadOub} P.~M.~Gadea and 
J.~A.~Oubi\~na, Reductive homogeneous pseudo-Riemannian manifolds, Monatsh.\ Math., {\bf 124}(1) (1997), 17--34. 

\bibitem{Gon} J.~C.~G\'onzalez-D\'avila, Harmonicity and minimality of distributions on Riemannian manifolds via the
intrinsic torsion, Rev.\ Mat.\ Iberoam.\ {\bf 30}(1) (2014), 247--275. 

\bibitem{Hel} S.~Helgason, Differential geometry, Lie groups and symmetric spaces, Academic Press, New York, 1978.

\bibitem{Ike} A.~Ikeda, Formally self adjointness for the Dirac operator on homogeneous spaces, Osaka J. Math.\ {\bf 12} (1975), 173--185. 

\bibitem{KobNom} S.~Kobayashi and K.~Nomizu, Foundations of differential geometry, II, Interscience Publishers, New York, 1969.

\bibitem{KowTri} O.~Kowalski and F.~Tricerri, Riemannian manifolds of dimension $n \leq 4$, Conferenze del Seminario di
Matematica Univ.\ di Bari, vol.\ 222, Laterza, Bari, 1987.

\bibitem{LawMic} H.~B.~Lawson and M.~L.~Michelsohn, Spin Geometry. Princeton Univ. Press, Princeton, 1990.

\bibitem{TriVan} F.~Tricerri and L.~Vanhecke, Homogeneous structures on Riemannian manifolds, Cambridge Univ.\ Press, Cambridge, UK, 1983.

\bibitem{Wol} J.~A.~Wolf, Spaces of constant curvature, 6th edn, AMS Chelsea Publishing, Providence, RI, 2011.

\end{thebibliography}
\end{document}